\documentclass[12pt]{article}

\usepackage{amsfonts}
\usepackage{amssymb}

\usepackage[T1, T2A]{fontenc}

\usepackage[cp1251]{inputenc}

\usepackage{lmodern}

\usepackage{amsthm} 

\usepackage[english]{babel}

\makeatletter
\@addtoreset{equation}{section}
\makeatother

\newtheorem{Theorem}{Theorem}   

\newtheorem{lemm}[Theorem]{Lemma}       
\newtheorem{prop}[Theorem]{Proposition} 
\newtheorem{state}[Theorem]{Statement}

\newtheorem{hypothesys}[Theorem]{Hypothesis}

\newcommand{\nqz}[1]{{\mathbb {#1}}}    
\newcommand{\goth}[1]{{\mathfrak {#1}}}   

\begin{document}
 
\title{
Artinian bimodule with quasi-Frobenius bimodule of translations }

\author{A.A.Nechaev\thanks{This article is dedicated to the memory of A.A.Nechaev}, V.N.Tsypyschev\thanks{Corresponding author: Vadim N.Tsypyschev(Tzypyschev), e-mail address: tsypyschev@yandex.ru}}

\date{}
    
\maketitle

\begin{abstract}
For bimodule  $_A M_B$ we intoduce notion of  \emph{  bimodule of translations}
 $_C M_Z$, where $C$ is a quotient ring of a \emph{Schneider} ring   of $_AM_B$  and $Z$ is its center. 

We investigate mutual relationship of two cases :

(а) bimodule  $_A M_B$  is a Morita's Artinian duality context and 

(b) bimodule  $_C M_Z$ is a Morita's Artinian duality context.
	
	Let's note that sometimes   bimodule of translations is called  \emph{ canonical bimodule } and $C$ is called \emph{multiplication } ring.

\end{abstract}

\section{Introduction }

Main results of this article previously were presented at Workshop  
\cite{NechaevZepeschThesysisQF2000}. 
Area of this investigation had  arose in 1990-s   when A.A.Nechaev рas had attempts to generalize the notion of linear recurrent sequence over commutative ring to the cases of linear recurrences over non-commutative ring, module, bimodule  
\cite{NechaevLRSQF1993,NechaevLRSQF1995,NechaevLRSQFPAM1995,MikhalevNechaevLRSQF1996,NechaevdeGru}.

Let's note immediately that the necessity for the ring, module, bimodule to be \emph{quasi-Frobenius} \cite{Azumaya} was established promptly.
Further one of the possible ways to determine linear recurrences over  non-commutative ring, module or bimodule was this one \cite{NechaevZepeschMSSU1995}.

Let  ${}_AM_B$ be an arbitrary bimodule.
\emph{ By the left translation  \/}
\cite{NechaevZepeschMSSU1995}
generated by the element  $a \in A$
is called a natural map 
$\hat a: M \to  M$ defined by equality : $\forall m \in M$ $\hat a
(m)=am$.
\emph{ Analogously  the right translation \/}
\cite{NechaevZepeschMSSU1995} generated by the element  $b \in B$
is a natural map 
$\check b$ defined by the equality : $\forall m \in M$ $\check b(m)=mb$.
It is evidently that 
$\hat a \in \mathop{\rm End}(M)=\mathop{\rm
End}(M_{\nqz{Z}})$ and 
$\check b \in \mathop{\rm
End}(M)$.
Subring 
$\hat A=\{ \hat a \; | \; a \in A \}$
of the ring  $\mathop{\rm End}(M)$ 
is called \emph{ ring of left translations.\/}
Respectively ring 
$\check B =\{ \check b
\; | \; b \in B \}$ is called 
\emph{ ring of right translations.\/}

  Let's remember that rings  $(K,+,*)$ and  $(R,\oplus,\otimes)$ are called \emph{  inversely  isomorphic or  anti-isomorphic \/} \cite{Kasch},
if ring's isomorphism takes place: $K^{op} \cong
R$ where  $K^{op}=(K,+,\tilde *)$ and operation  $\tilde *: K\times K\to K$
is defined according to the rule: $k_1\tilde *k_2=k_2*k_1,\ k_1,k_2\in K$.

It is easy to see that if  ${}_AM$  is a faithful module then rings $A$ and  $\hat A
$  are isomorphic  and if  $M_B$ is a faithful module then rings  $B$ and 
$\check B$  are inversely  isomorphic .

\begin{lemm}
\label{l1(a)}
For arbitrary bimodule  ${}_A M_B$ this equality takes place: 
$\hat A \cap \check B=\mathop{\rm \bf Z}(\hat A) \cap
\mathop{\rm \bf Z}(\check B)$. Here $\mathop{\rm\bf Z}(\hat A)$ and $\mathop{\rm\bf Z}(\check B)$ denote centers of corresponding rings. 

\end{lemm}

\begin{proof} Because of the associativity  $(am)b=a(mb)$ elements of the rings  $\hat A$ and  $\check B$ are  pair-wise commutative. Hence arbitrary element 
 $z \in \hat A \cap \check B$
lies in  $\mathop{\rm \bf Z}(\hat A)$ and in  $\mathop{\rm \bf Z}(
\check B)$ concurrently,  i.e. 
$\hat A \cap \check B \subset \mathop{\rm \bf Z}(\hat A) \cap
\mathop{\rm \bf Z}(\check B)$. 
Inverse inclusion is evident. 
\end{proof}

Commutative ring $Z=\hat A \cap \check B$ ( Lemma \ref{l1(a)}) is called  \cite{NechaevZepeschMSSU1995}
\emph{ common center of rings  $A$ and  $B$  in relation to  bimodule  $M$. \/}
Let's denote by  $\phi: \hat A \to A$
isomorphism constructed according to the rule:  $\hat a \stackrel{\phi}{
\mapsto} a$ and by 
$\psi:
\check B \to B$
inverse isomorphism of rings 
$\check B$ and  $B$  constructed by the rule: 
$\check b \stackrel{\psi}{\mapsto} b$.
Then  $\phi(Z) < \mathop{\rm \bf Z}(A)$, $\psi(Z)<\mathop{\rm \bf
Z}(B)$.

Besides that to  the set  $M$ may be correctly assigned structure of  
$(Z,Z)$-bimodule  in such way that these identities hold: 
$z \cdot m=\phi(z)m$, $m \cdot z=m\psi(z)$,
$zm=mz$.

 $(Z,Z)$-bimodule 
$\hat A \mathop{\otimes}\limits_Z \check B$ has an associative  multiplication operator 
\cite[Proposition 9.2(a)]{Pierce} defined by the rule:
$(\hat a_1 \otimes \check b_1 )
(\hat a_2 \otimes \check b_2 )=
\hat a_1 \hat a_2 \otimes \check b_1 \check b_2$.
Herewith if  $A$ and  $B$ are rings with identity then 
$1_{\hat A \otimes \check B}=1_{\hat A} \otimes 1_{\check B}$.
Now  $\hat 
A \mathop{\otimes}\limits_Z \check B$ is a ring and $M$  is a 
\cite[Proposition 10.1(a)]{Pierce}  left  $\hat A
\mathop{\otimes}\limits_Z \check B$-module with multiplication according to the rule:  $(\hat a \otimes \check b)m=\phi(\hat a)m
\psi(\check b)$.
If modules  ${}_AM$ и $M_B$  are unitary  then 
$1_{\hat A \otimes \check B}=\epsilon$ is an identity mapping of   $M$.

Further we will need this generalization of 
 \cite[$\S10.1$,
Proposition ]{Pierce}:

\begin{prop}
 Let  $A, B$ are  $Z$-algebras. Any 
$(A,B)$-bimodule  $M$
is a left  $A\mathop{\otimes}\limits_ZB^{op}$-module with scalar multiplication defined according to the rule: 
$(x\otimes y)u=(xu)y=x(uy)
$
for  $x\in A, y\in B$, $u\in M$.

Conversely, any left  $A\otimes B^{op}$-module  is an $(A,B)$-bimodule with operations: 
$
xu=(x\otimes1_B)u,\ \ \ uy=(1_A\otimes y)u.
$
If  $M$ and  $N$ are  $(A,B)$-bimodules  then 
$
\mathop{\rm Hom}_{A-B}(M,N)=\mathop{\rm Hom}_{A\otimes B^{op}}(
{}_{A\otimes B^{op}}M, {}_{A\otimes B^{op}}N).
$
\end{prop}

Ring $C$ is called \emph{ ring of translations  of bimodule ${}_A M_B$\/}
\cite{NechaevZepeschMSSU1995} if it is generated in the ring 
 $\mathop{\rm End}(M)$ by the  union of rings  $\hat A$ and 
$\check B$.

It takes place 
\begin{state}
\label{st2}
Mapping $\lambda: \hat A \mathop{\otimes}\limits_Z
\check B \to C$ satisfying  to the rule: 
$\lambda\left(\sum\limits_{(i)} \hat a_i \otimes \check b_i \right)
=\sum\limits_{(i)} \hat a_i \check b_i$
is an epimorphism of rings and 
 $\ker \lambda =\{ \delta \in
\hat A \mathop{\otimes}\limits_Z \check B \  | \  \delta M=\theta \}$
where  $\theta$ is an identity element of  $M$.
\end{state}

\begin{proof}  The first statement follows from element-wise commutativity of rings  $\hat A$ and  $\check B$. The second statement follows from the fact that  ${}_C M$  is a faithful module.  
\end{proof} 

Previously \cite{NechaevZepeschMSSU1995}
bimodule  ${}_C M_Z$
was called 
\emph{ canonical bimodule of bimodule  ${}_A M_B$. \/}
            
Further we instead or denomination   \emph{ canonical bimodule \/}
will use denomination 
\emph{  bimodule of translations \/}.

 Now   \emph{linear recurrent sequence } over bimodule  $_AM_B$ may be defined in such way   \cite{NechaevZepeschMSSU1995}: it is a mapping   $\mu: {\mathbb N}_0\to M$ for which there exists a unitary polynomial  $c(x)=x^m-
\sum_{j=0}^{m-1}c_jx^j \in C[x]$ with property:  $\forall i\in{\mathbb N}_0 \ \mu(i+m)=
\sum_{j=0}^{m-1} c_j \left(\mu(i+j)\right)$. 

Herewith to generalize classical properties  of  linear recurrences onto newly arisen case it is necessary for bimodule  $_CM_Z$  to be \emph{quasi-Frobenius.}

Let's remember that according to definition 
\cite{Azumaya} left-faithful and right-faithful bimodule $M$ over rings $A$ and  $B$ with identities is called \emph{quasi-Frobenius} if for every maximal left  ideal 
 $I$ of ring  $A
$
\emph{right annihilator   \/}
$r_M(I)=\{ m \in M \; | \; Im=\theta \}$
is either irreducible  $B$-module either equal to neutral element of $M$ 
and for every maximal right ideal  $J$ of ring  $B$
\emph{left annihilator   \/}
$l_M(J)=\{ m \in M \; | \;
mJ=\theta \}$
is either irreducible  $A$-module either equal to $\theta$. 
Here and after $\theta$ is a neutral element of the group  $(M,+)
$.

However it is a problem  that there were not known any facts about  relations between quasifrobeniusness of bimodule  $_CM_Z$  and qusifrobeniusness of bimodule 
$_AM_B$.
Initially was formulated  
\begin{hypothesys}[\cite{NechaevZepeschMSSU1995}]
If bimodule  $_AM_B$ is a Morita's Artinian duality context then  bimodule of translations  $_CM_Z$ is also a Morita's Artinian duality context. 
\end{hypothesys}

However during multiyear process of attempts to prove this Hypothesis it become clear that  it's in general case wrong  (Theorem \ref{TH4+1}) and moreover under additional conditions inverse implication holds  (Theorem  \ref{TH4}).

Resulting in these facts approach to definition of linear recurrences over bimodules suggested by V.L.Kurakin was adopted  \cite{KuMiN2000}. Theoretically this approach was based on  results of
\cite{KurakinNechaevBiQF2001}.

Despite this results of this article are not negligible  and were continued: notion of \emph{matrix linear recurrences} is based on the Theorem  \ref{TH4}  
\cite{KuMiNeZ,ZepeschMLCG} and the notion of \emph{skew linear recurrences} is based on Theorem  \ref{TH4+1}  \cite{NechaevSkew}.

Let's now get to outline and prove main results of this article. In order of independence and completeness we will remember necessary definitions  
 \cite{Kasch}.
Module  ${}_A M$ is called  \emph{ injective  \/} if 
for every monomorphism  $\alpha: {}_A K \to {}_A L$  and every homomorphism  $\phi: {}_A K \to {}_A M$ there exists homomorphism  
$\psi: {}_A L \to {}_A M$ with property $\phi=\psi\alpha$.
It is equivalent  \cite[Theorem  5.3.1(a)]{Kasch} to the fact that for every monomorphism  $\xi: {}_A M \to {}_A N$
module 
$\xi(M)$
allocated in  $N$ as a direct summand.  According to \cite[Theorem 
5.6.4]{Kasch} for every  $A$-module  $L$ there exists unique up to isomorphism \emph{ injective hull \/}
$\mathop{\rm \bf I}(L)
$ i.e. minimal according to order of inclusion injective module containing $L$.
Module  ${}_AU>{}_AL$ is called  \emph{ essential extension   \/} of module  ${}_AL$ if for every submodule  $V$  of module  ${}_AU$
from condition  $L\cap V=\theta$  follows that  $V=\theta$.
It is known  \cite[Теорема 5.6.6]{Kasch} that injective hull  $\mathop{\rm\bf I}(L)$ of module  $L$
is a maximal essential extension  of  $L$.

Also we have to remember the notion of \emph{ Brauer group of the field $F$} 
\cite[$\S 12.5$]{Pierce}.

Let  $F$  be a field.  Let's denote by  $\goth{B}(F)$ class of all finite-dimensional simple $F$-algebras  $A$ with property  $\mathop{\rm \bf Z}(A)=F$ i.e. class of all finite-dimensional central simple $F$-algebras. 

\emph{  If algebras  $A$ and  $B$ contains in  $\goth{B}(F)$ then 
\cite[$\S12.4$, Proposition  b(i)]{Pierce} $F$-algebra 
$A\mathop{\otimes}\limits_FB$ contains in  $\goth{B}(F)$.}

\emph{Besides that 
\cite[$\S12.5$, Lemma ]{Pierce} following conditions are equivalent: }

\emph{(i) basic algebras of  $A$ и $B$ are isomorphic;}

\emph{(ii) there exists an algebra with division  (body) $D\in\goth{B}(F)$ and naturals  $m$ and  $n$ such that  $A\cong D_{n,n}$ and  $B\cong D_{m,m}$;}

\emph{(iii) there exists naturals  $r$ and  $s$ such that  $A\otimes F_{r,r}
\cong B\otimes F_{s,s}$. \/}

Algebras  $A, B\in\goth{B}(F)$, underlying this conditions are called 
 \emph{ equivalent. \/}
Equivalence class of algebra   $A$ is denoted as 
 $[A]$.

\emph{ The set  $\mathop{\rm\bf B}(F)=\{[A]\ |\ A\in\goth{B}(F)\}$ is 
\cite[$\S12.5$, Proposition  a]{Pierce} an Abel group with operator 
$[A][B]=[A\otimes B]$, neutral element  $[F]$ and the operation of taking inverse element $[A]^{-1}=[A^{op}]$.\/} The group  $\mathop{\rm\bf B}(F)$ is called 
 {\em Brauer group \/} of the field  $F$. \emph{ Every class in group 
$\mathop{\rm\bf B}(F)$ is represented by  \cite[$\S12.5$, Proposition 
b(ii)]{Pierce}  algebra with division which is unique up to isomorphism .\/}
If Brauer group of field $F$ is trivial then every central simple $F$-algebra is a full matrix algebra over $F$.  
\emph{ If the field  $F$ is algebraically closed then its Brauer group is trivial \/} \cite[$\S12.5$, Corollary]{Pierce}. Besides that Brauer group of the finite field is trivial too.

By \emph{ socle  \/}
$\goth{S}({}_A M)$ of the left  $A$-module  $M$ is called 
\cite[Chapter  IV]{Jacobson}
the sum of all irreducible submodules of the  $A$-module  $M$.

The main result of this work consists in following: 
\begin{Theorem}
\label{TH4}

Let bimodule        
 ${}_A M_B$ be a faithful as a left $A$-module and as a right 
$B$-module simultaneously,  bimodule of translations   ${}_C M_Z$
of bimodule  ${}_A M_B$ is quasi-Frobenius, $Z$ is a local Artinian ring, 
$Q=\mathop{\rm\bf I}(\bar Z_Z)$ is an injective hull of the unique irreducible  $Z$-module  $\bar Z$.

Then: 
 
I. (a) Ring's isomorphism takes place:  $C \cong Z_{n,n}$. 

(b) Bimodule's isomorphism takes place:  ${}_CM_Z \cong {}_{Z_{n,n}}Q^{(n)}_Z$.

(c) Rings  $A$, $B$ are primary Artinian left-side and right-side simultaneously rings with identity  and factor-rings  $\bar A=A/J(A)$,
$\bar B=B/J(B)$
of rings  $A$ and  $B$ respectively when activated modulo  Jacobson's radical are equivalent elements of the set  $\goth{B}(\bar Z)$.
         
(d) Bimodule  ${}_{\bar A}\goth{S}({}_CM)_{\bar B}$ is quasi-Frobenius. 

II. Bimodule  ${}_AM_B$ is quasi-Frobenius if and only if the left socle  $\goth{S}({}_AM)$  and the right socle  $\goth{S}(M_B)$
of bimodule  ${}_AM_B$ are equal to $\goth{S}({}_CM)$.
                             
III. Under additional condition of the form: Brauer group  $\mathop{\rm \bf B}(\bar Z)$ of the field  $\bar Z$ is trivial 
bimodule   ${}_AM_B$ is quasi-Frobenius and moreover for some naturals  $r,s \in \nqz{N}$ such that  $n=r\cdot s$ isomorphisms of rings and bimodules respectively takes place: 
                                         
(a) $A \cong Z_{r,r}$, $B \cong Z_{s,s}$,

(b) ${}_A M_B \cong {}_{Z_{r,r}} {Q_{r,s}}_{Z_{s,s}}$.
\end{Theorem}

Here we denote by 
${}_{Z_{r,r}} {Q_{r,s}}_{Z_{s,s}}$  the set of orthogonal matrices of dimensions
$r\times s$  filled by elements of the set  $Q$ and usual matrix multiplication upon elements of 
$Z_{r,r}$ and  $Z_{s,s}$ respectively. 
It is evident that 
${}_{Z_{r,r}} {Q_{r,s}}_{Z_{s,s}}$
is a bimodule.

Conversion of the point  III of Theorem  \ref{TH4} is wrong even in the finite case. 
Appropriate example may be constructed in class of so called 
$GEO$-rings  \cite{NechaevMS}.
By definition the ring  $S$ is called \emph{ Galois-Eisenstein-Ore ring \/} (or \emph{ $GEO$-ring \/}) if it is finite completely primary ring of principal ideals i.e. the ring  $S$
contains unique (one-side) maximal ideal  $\goth{p}(S)$ and every one-side ideal of ring  $S$ is principal. 
\emph{Commutative\/} $GEO$-ring is called \emph{Galois-Eisenstein ring or $GE$-ring} \cite{NechaevMS}.
                                      
In arbitrary  $GEO$-ring  $S$ of characteristic $p^d$ contains 
\emph{ Galois subring \/} $R=GR(p^d,r)$ \cite{McDonald, Radghavendran}, $q=p^r$.
All such rings are conjugated in  $S$. Ring  $R$ is called \emph{coefficients ring \/} of  $S$.

\begin{Theorem}
\label{TH4+1}

(a) For arbitrary  $GEO$-ring  $S$ bimodule ${}_SS_S$  is a quasi-Frobenius. 
                   
(b)
Let  $S$ be a  $GEO$-ring with a coefficients ring  $R=GR(p^d,r)$ and in addition  $d>1$ and
$(r,p)=1$.
Then  bimodule of translations 
 ${}_CS_Z$ of quasi-Frobenius 
$S$-bimodule 
$S$ is a quasi-Frobenius bimodule if and only if 
$S$ is a  $GE$-ring. 

\end{Theorem}

\section{Preliminaries }
      
To prove Theorems above it is necessary to remember a lot of definitions and facts connected firstly with different characterizations of quasi-Frobenius bimodules. 
Most deeply and fully quasi-Frobenius bimodules are investigated in the case when rings  $A$
and  $B$ are \emph{Artinian   \/} left-side and right-side respectively, i.e. it satisfy descending chain conditions of left (respectively, right) ideals.  
Such quasi-Frobenius bimodules in the monography  \cite{Faith} are called 
\emph{ (Morita's) Artinian duality context \/}.
In this case  \cite[Theorem  6(iv),(v)]{Azumaya}
modules  ${}_A M$ and  $M_B$ are finitely generated. 

Let's remember other well-known results about quasi-Frobenius bimodules. 

Let  $A$ be a left Artinian ring with unit  $e$,
$J(A)$ be a Jacobson radical of ring $A$.
Then  $J(A)$ (\cite[Corollary  III.1.1]{Jacobson}) is a maximal nil-potent two-sided ideal of ring  $A$ containing every one-side nil-ideal of ring $A$. 
It is known \cite[Chapter  III]{Jacobson} that  $\bar A=A/J(A)$
is a direct sum of pair-wise orthogonal simple rings  
$\bar A_1,\ldots,\bar A_t$ with units  $\bar e^{(1)},\ldots, \bar
e^{(t)}$ respectively. 
Idempotent  $\bar e_l$, $l \in
\overline{1,t}$ of ring  $\bar A_l$
is called 
\cite{Jacobson} \emph{ primitive \/}
if left ideal  $\bar A \bar e_l$ of ring 
$\bar A$ is irreducible as a left  $\bar A$-module. Then 
$\Delta_l=\mathop{\rm End}\left({}_{\bar   A}    \bar    A\bar
e_l\right)$ is a body isomorphic to the ring 
$\bar e_l \bar A \bar e_l=\bar e_l A_l \bar e_l$ and ring 
$\bar A_l$ is a full matrix algebra of finite degree 
$k_l$ over body $\Delta_l$.
Besides that modules  $\bar A_l \bar e_l$,$l=\overline{1,
t}$ are all, up to isomorphism, irreducible 
$A$-modules. 

Analogously idempotent  $e_l$, $l \in
\overline{1,t}$, of ring  $A$
is called 
\cite{Jacobson} {\em primitive \/}
if left ideal  $\bar A \bar e_l$ of ring 
$\bar A$ is irreducible left $\bar A$-module.

Idempotents  $\bar e_j, \bar e_l \in \bar A$ are called isomorphic if isomorphic left 
 $\bar A$-modules $\bar A\bar e_j$ and  $\bar A
\bar e_l$.
Analogously idempotents  $e_j,  e_l \in  A$ are called isomorphic if isomorphic left  $A$-modules  $Ae_j$ and $Ae_l$.

Let 
$\bar e_l$, $l=\overline{1,t}$, are all non-isomorphic idempotents of ring  $\bar A$.
                          
For every 
$l=\overline{1,t}$ there exists idempotent  $e_l
\in \bar e_l$, such that elements  $e_l$, $l=\overline{1,t}$,
are all non-isomorphic primitive idempotents of ring $A$, and there takes place 
\cite[Proposition III.8.5]{Jacobson}
the expansion unit  $e$ of ring  $A$ into the sum of primitive orthogonal idempotents 
\begin{equation}
\label{eq0}
e=\sum\limits_{l=1}^{t}  \sum_{i=1}^{k_l}
e_{l,i},
\end{equation}
where every idempotent   $e_{l,i}$ is isomorphic to idempotent $e_l$
(i.e. by definition   ${}_A Ae_{l,i} \cong
{}_A Ae_{l}$).

Decomposition  (\ref{eq0}) generates decomposition of the ring  $A$ into the direct sum of undecomposable left ideals: 
\begin{equation}
\label{eq0.0}
A=\sum\limits_{l=1}^{t} \oplus  \sum_{i=1}^{k_l}
\oplus Ae_{l,i}
\end{equation}
                                                      
Let's denote 
$e^{(l)}=\sum\limits_{i=1}^{k_l} e_{l,i}
$, $l=\overline{1,t}$.
It is known  \cite[Theorem  III.9.2]{Jacobson} that additive group of the ring  $A$ is represented as a direct sum of Abel groups 
\begin{equation}
\label{eq0.0.0}
A=\sum\limits_{l=1}^{t} \oplus
A_l \oplus {\cal U},
\end{equation}
where 
$A_l=e^{(l)}Ae^{(l)}$,
$l=\overline{1,t}$,
${\cal U}=\sum\limits_{l \ne \nu} e^{(l)} A e^{(\nu)}$.
Herewith  $A_l
=S^{(l)}_{k_l,k_l}$,
$l=\overline{1,t}$ are primarily rings with identities 
$e^{(l)}$ i.e. every of it is a  full matrix ring over local left Artinian ring  $S^{(l)}$ wherein it's a preimage of ring $\bar A_l$ when activating modulo  $J(A)$,
${\cal U}$ is a  subgroup of group  $(J(A),+)$.
Moreover herewith 
\begin{equation}
\label{eq0.0.0.0}
J(A)=\sum\limits_{l=1}^t J(A_l) \oplus {\cal U}.
\end{equation}
Further rings  $A_l$, $l=\overline{1,t}$,
we will call  {\em primarily components  \/} of ring  $A$.
If  $t=1$ then  \cite[Theorem  III.8.1]{Jacobson} $A=S_{k,k}$  is a primary ring. 

Let's remember that socle 
$\goth{S}({}_A M)$ of the left  $A$-module  $M$
\cite[Chapter  IV]{Jacobson}
is a completely reducible  left $A$-module satisfying condition: 
$\goth{S}({}_A M)=r_M(J(A))$.
Sum 
$\goth{S}^{(V)}({}_A M)=\sum_{i\in I}\tilde V_i$
of all irreducible submodules 
$\tilde V_i, i\in I$,
of module  ${}_A M$  isomorphic to module  ${}_AV$ is called 
\emph{ homogeneous component of socle  $\goth{S}({}_A M)$
belonging to irreducible left $A$-module $V$. 
\/}
                                               
If  ${}_AM_B$ is some bimodule and  $V$ is a some left $A$-module then $B$-module 
$V^*_B=\mathop{\rm Hom}\nolimits_A(V,M)$ is called \emph{ a right dual to   $V$  with respect to bimodule  $M$\/}
\cite[Chapter  IV]{Jacobson}. 
Relatively  by \emph{left dual to the right 
$B$-module  $W$ with respect to  $
{}_AM_B$\/}
is called  $A$-module 
${}_AW^*=\mathop{\rm Hom}\nolimits_B(W,M)$.

If 
${}_A M_B$ is a quasi-Frobenius bimodule, 
${}_A V$
and $W_B$ are irreducible modules then $V_B^*$ и ${}_AW^*$ are also irreducible (either are non-zero).
Moreover 
\cite[Theorem  1]{Azumaya} these equalities take place: 
$\goth{S}^{(V)}({}_A M)=\goth{S}^{(V^*)}(
M_B)$, $\goth{S}^{(W)}(M_B)=\goth{S}^{(W^*)}({}_A M)$,
where  $V^*_B=\mathop{\rm Hom}\nolimits_A
(V,M)$,
${}_AW^*=\mathop{\rm Hom}\nolimits_B(W,M)$.
Hence  $\goth{S}({}_A M)=\goth{S}(M_B)$.

Left $A$-module 
$M$ is called \emph{ distinguished  \cite{Azumaya} \/} if it contains an isomorphic image of every irreducible  $A$-module or, what's  equivalent, for every irreducible left $A$-module $V$ 
$\goth{S}^{(V)}({}_A M) \ne \theta$.
Left 
$A$-module  $U$ is called 
\emph{ minimal (in  $M$) \cite{Azumaya} \/} if every submodule of ${}_A M$ contains $U$. 
If  minimal submodule exists it is a unique irreducible submodule of  ${}_A M$.

Now we can provide necessary characterizations of quasi-Frobenius bimodules: 

\begin{Theorem}
\label{th1}
Following conditions are equivalent: 

(1) \cite[Theorem  6(iii)]{Azumaya} Bimodule ${}_A M_B$ is a quasi-Frobenius, rings  $A$ and  $B$ are left Artinian and right Artinian respectively. 

(2)
\cite[Theorem  6(i)]{Azumaya}
Ring  $A$ is a left Artinian, left  $A$-module 
$M$
is an injective, finitely generated and distinguished, ring  $\check
B$ 
is an endomorphism ring of module  ${}_A M$.

(3)\cite[Theorem  23.25(6)]{Faith} Every left ideal  $I$ of ring  $A$,
every right  $B$-submodule  $W$ of module  $M$ are satisfies to 
\emph{ \bf annihilator correspondences:  \/}
\begin{equation}
\label{eq1}
\begin{array}{cc}
I=l_A(r_M(I)), & W=r_M(l_A(W)), \\
\end{array}
\end{equation}

every right  ideal  $J$
of ring  $B$, every left  $A$-submodule of   $M$
satisfy to \emph{ \bf annihilator correspondences: \/}
\begin{equation}
\label{eq1+1+1}
\begin{array}{cc}
J=r_B(l_M(I)), & V=l_M(r_B(V)).
\end{array}
\end{equation}

(4) \cite[Theorem  12]{Azumaya} Bimodule  ${}_A M_B$ is a faithful, 
$A$ is a left Artinian ring,  $B$ is a right Artinian ring, $A$ and $B$ have the same number  $t$ of 
non-isomorphic  primitive idempotents, namely: 
$e_1,
\ldots,e_t$, $f_1,\ldots,f_t$ respectively, which under necessary re-enumeration satisfy conditions: 

I. For every  $l \in
\overline{1,t}$
left  $A$-module 
$Mf_l$
contains minimal  $A$-submodule isomorphic to 
 $\bar A \bar e_l$,
an image of 
 $Ae_l$ activating modulo 
 $J(A)e_l$;

II. For every  $l\in
\overline{1,t}$
right  $B$-module  $e_l M$
contains minimal 
 $B$-submodule isomorphic to 
$\bar f_l \bar B$,
an image of  $f_l B$ activating modulo 
$f_l J(B)$;

(5)
\cite[Theorem  6(ii)]{Azumaya}
Ring 
 $B$ is right Artinian, right $B$-module  $M$ is injective, finitely generated and distinguished, ring  $\hat A$ is an endomorphism ring of module  $M_B$.
\end{Theorem}
           
If 
 ${}_A M_B$  is a quasi-Frobenius bimodule satisfying to conditions of Theorem \ref{th1} then 
 \cite[Theorem  7]{Azumaya} there exist one-to-one Galois correspondences between the set of two-sided ideals 
 $I$  of ring  $A$,
the set of  $(A,B)$-subbimodules  $N$ of bimodule  $M$,
and the set of two-sided ideals  $J$ of ring  $B$, which are  established  by equalities: 
$$N=r_M(I),\ \  J=r_B(N);$$
$$N=l_M(J),\ \  I=l_A(N).$$
Herewith  $N$ is a quasi-Frobenius  $(A/I,
B/J)$-bimodule. In particular,  $\goth{S}(M)$ is a quasi-Frobenius  $(A
/J(A),B/J(B))$-bimodule.

Further we will use this 
\begin{state}[A.A.Nechaev, 1996г.]
\label{st1}
Left-faithful and right-faithful  bimodule  ${}_A M_B$ over left Artinian and right Artinian correspondingly rings with units is quasi-Frobenius if and only if left socle 
$\goth{S}({}_A M)$ and right socle  $\goth{S}(M_B)$ of  $M$ coincides and 
$(\bar A,\bar B)$-bimodule 
\begin{equation}
\label{eq0+1}
\goth{S}(M)=\goth{S}({}_A M)=\goth{S}(M_B)
\end{equation}
is quasi-Frobenius. 
\end{state}
 
\begin{proof} For quasi-Frobenius bimodule ${}_A M_B$ we have equality: 
 $\goth{S}({}_A M)=\goth{S}(M_B)$
\cite[Theorem  1]{Azumaya}.
Because  ${}_A M_B$ is satisfying to conditions of Theorem  \ref{th1} we have that 
\cite[Theorem  7]{Azumaya} $\goth{S}(M)$ is a quasi-Frobenius  $(\bar
A,\bar B)$-bimodule. 

Conversely let equalities  (\ref{eq0+1}) take place and bimodule 
${}_{\bar A}\goth{S}(M)_{\bar B}$ is a quasi-Frobenius.  Then bimodule  ${}_{\bar A} \goth{S}(M)_{\bar B}$ satisfies conditions of point 4 of Theorem  \ref{th1} and according to it notations these equalities take place: 
 $\goth{S}(e_l M_B)=\goth{S}(\bar e_l
\goth{S}(M)_{\bar B})$ and  $\goth{S}({}_A Mf_l)=\goth{S}({}_{\bar
A} \goth{S}(M) \bar f_l)$. Hence if conditions  of point 4 of Theorem  \ref{th1} take place  relatively to bimodule  ${}_{\bar A} \goth{S}(M)_{
\bar B}$ then it implies satisfying of the same conditions in relation to bimodule  ${}_A M_B$.
\end{proof}

\begin{lemm}
\label{l1(b)}
For Artinian duality context 
${}_A M_B$ this equality 
takes place:
$\mathop{\rm \bf Z}(\hat A)=
\mathop{\rm \bf Z}(\check B)$.
\end{lemm}

\begin{proof}
According to Theorem  \ref{th1}(2), $\hat A=\mathop{\rm End}(M_B)$.
Because an arbitrary element  $z \in \mathop{\rm \bf Z}(\check B)$
contains in $\mathop{\rm End}(M_B)$ then  $\mathop{\rm \bf Z}(\check B)
\subset \hat A$ what follows
$\mathop{\rm\bf Z}(\check B)\subset\mathop{\rm\bf Z}(\hat A)$.
Converse inclusion may be stated by the same method. 
\end{proof}

  It is also true                                                            
\begin{state}
\label{st3}
If 
$M$ is a bimodule over primary left-Artinian and right-Artinian  rings with identities respectively
 $A=S_{k,k}$ and  $B=T_{l,l}$ of the form:  
$$
{}_AM_B={}_{S_{k,k}}{W_{k,l}}_{T_{l,l}},
$$
then bimodule  ${}_AM_B$ is a quasi-Frobenius if and only if bimodule  ${}_SW_T$ is a quasi-Frobenius.
\end{state}

\begin{proof}
Let suppose that bimodule  ${}_AM_B$ is a quasi-Frobenius. 
Let  $e$ be a primitive idempotent of ring  $A$. Then according to point (4) of Theorem  \ref{th1}      right  $B=T_{l,l}$-module  $eM=W^{(l)}$ contains a minimal submodule. 
Let suppose now that right  $T$-module  $W$ does not contain minimal submodule.
It means that module  $W_T$
contains at least two diffrerent irreducible submodules  $U$ and  $V$.
Then module  $W^{(l)}_{T_{l,l}}$ contains at least two different irreducible submodules 
 $U^{(l)}_{T_{l,l}}$ and 
$V^{(l)}_{T_{l,l}}$. 
It is a contradiction to existence in  $W^{(l)} {}_{T_{l,l}}$ of minimal submodule. Hence right module  $W {}_T$ also contains minimal submodule. 

Analogously may be shown that existence of minimal submodule in  ${}_{S_{k,k}} W^{(k)}$ implies existence of minimal submodule in  ${}_SW$.

Hence according to point (4) of Theorem  \ref{th1} bimodule  ${}_SW {}_T$  is a quasi-Frobenius.

The converse implication may be proven by the same  arguments  in reverse order.   
\end{proof}

\section{The Proof of Theorem \ref{TH4}}

For convenience of reading let us duplicate the text of Theorem  \ref{TH4}.

\begin{Theorem}
\label{th4+2}
       
Let bimodule 
 ${}_A M_B$ be faithful as a left  $A$-module  and as a right 
$B$-module together,
bimodule of translations  ${}_C M_Z$
of  ${}_A M_B$ is a quasi-Frobenius,  $Z$ is local Artinian ring, 
$Q=\mathop{\rm\bf I}(\bar Z_Z)$ is an injective hull of unique irreducible  $Z$-module  $_Z\bar Z$.

Then:                  
 
I. (a) Isomorphism of rings takes place:  $C \cong Z_{n,n}$. 

(b) Isomorphism of bimodules takes place:  ${}_CM_Z \cong {}_{Z_{n,n}}Q^{(n)}_Z$.

(c) Rings  $A$, $B$ are primary left-Artinian and right-Artinian rings with identities, herewith factor-rings $\bar A=A/J(A)$,
$\bar B=B/J(B)$
of  $A$ and  $B$ respectively activated modulo Jacobson radicals are equivalent elements of the set 
 $\goth{B}(\bar Z)$.
         
(d) Bimodule  ${}_{\bar A}\goth{S}({}_CM)_{\bar B}$ is a quasi-Frobenius. 

II.Bimodule ${}_AM_B$ is a quasi-Frobenius if and only if the  left socle  $\goth{S}({}_AM)$ and the right socle  $\goth{S}(M_B)$ of 
 ${}_AM_B$  coincide with $\goth{S}({}_CM)$.
                             
III.If additional condition takes place: 
Brauer group  $\mathop{\rm \bf B}(\bar Z)$ of the field  $\bar Z$ is trivial, then bimodule 
 ${}_AM_B$  is a quasi-Frobenius and moreover for some   $r,s \in \nqz{N}$ such that  $n=r\cdot s$ isomorphisms of rings and bimodules respectively take place: 
                                         
(a) $A \cong Z_{r,r}$, $B \cong Z_{s,s}$,

(b) ${}_A M_B \cong {}_{Z_{r,r}} {Q_{r,s}}_{Z_{s,s}}$.
\end{Theorem}

\begin{proof}   
     \ \ \  I(a).\ 
Let  ${}_CM_Z$ be an Artinian duality context. Then according to Theorem 
 \ref{th1}(5) the right module $M_Z$ is an injective and hence it is a direct sum of some  $n\in\nqz{N}$ right modules  $Q_Z$:
\begin{equation}
\label{eq.th4.1}
M_Z\cong Q_Z^{(n)},
\end{equation}
where  $Q=\mathop{\rm\bf I}(\bar Z_Z)$  is an injective hull of unique irreducible  $Z$-module  $\bar Z=Z/J(Z)$.

By the same Theorem the equality take place:  $\hat C=C=\mathop{\rm End}(M_Z)$.
The existence of  isomorphism  (\ref{eq.th4.1}) implies  isomorphism: 
\begin{equation}
\label{eq.th4.2}
C\cong D_{n,n},
\end{equation}
where  $D=\mathop{\rm End}(Q_Z)$. Because of  
\cite[Proposition  23.33]{Faith}, \cite{Muller},
equality takes place:  $D=Z$ and isomorphism  (\ref{eq.th4.2}) takes a form: 
\begin{equation}
\label{eq.th4.3}
C\cong Z_{n,n}.
\end{equation}

I(b).\ From  (\ref{eq.th4.1}) and  (\ref{eq.th4.3}) follows that 
\begin{equation}
\label{eq.th4.3+1}
{}_CM_Z\cong{}_{Z_{n,n}}Q^{(n)}_Z.
\end{equation}
                                                          
I(c).\
From  (\ref{eq.th4.3}) follows that  $C$ is finitely generated left $Z$-module and finitely generated right  $Z$-module simultaneously. Because  $Z$ is a commutative Artinian ring then  $C$ is also Artinian ring ( and not only left-Artinian how  the Theorem  \ref{th1}(1) states ).

Under conditions of this Theorem left module  ${}_AM$ is a faithful. Hence rings  $A$
and  $\hat A$ are isomorphic. Because ring $\hat A$ is a subring of $C$ and hence contains ring  $Z$, so module  $\hat A_Z$ is a submodule of the finitely generated module 
 $C_Z$ over commutative Artinian ring  $Z$. Hence $\hat A$ is a two-sided Artinian ring and finitely generated  $Z$-module. Because of the isomorphism,  ring $A$ is also two-sided Artinian.

Analogously may be established that  $\check B$ is a two-sided Artinian ring and finitely generated $Z$-module and that ring  $B$ is also two-sided Artinian. 

Let's establish now the relation between Jacobson radical   $J(C)$ of ring  $C$  and Jacobson  radicals $J(\hat A), J(\check B)$ of rings  $\hat A$,   $\check B$ respectively.

Let  $\hat a\in J(\hat A)$. Then right ideal  $\hat aC$ generated by element  $\hat a$ in ring  $C$ is a nil-ideal. Indeed for arbitrary element  $c=\sum\limits_{(i)}\hat a_i\check b_i\in C$ because of the element-wise commutativity of rings  $\hat A$ and  $\check B$  the equality takes  place: 
\begin{equation}
\label{eq.th4.4}
(\hat a c)^{n(\hat A)}=\sum_{\left(i_1,\ldots,i_{n(\hat A)}\right)}
\hat a\hat a_{i_1}\cdots\hat a\hat a_{i_{n(\hat A)}}\cdot \check b_{i_1}
\cdots\check b_{n(\hat A)},
\end{equation}
where $n(\hat A)$ is a nil-potent index of ideal $J(\hat A)$. Because for arbitrary set  $\left(i_1,\ldots,i_{n(\hat A)}\right)$  the inclusion takes place: 
$$
\hat a\hat a_{i_1}\cdots\hat a\hat a_{i_{n(\hat A)}}
\in J(\hat A)^{n(\hat A)}=\hat 0,
$$
then  $(\hat a c)^{n(\hat A)}=\hat 0$.

Because every one-side nil-ideal of Artinian ring contains in Jacobson radical of it ring we have:  
 $\hat aC<J(C)$. Besides that,  $\hat a\in\hat aC$. So we have proved inclusion:  $J(\hat A)\subset J(C)$.

From other side the set $J(C)\cap \hat A$ is a two-sided nil-potent ideal of ring  $\hat A$. Hence, 
 $J(C)\cap\hat A\lhd J(\hat A)$. 

So we have a sequence of inclusions:  $J(\hat A)\subset J(C)\cap
J(\hat A)\subset J(C)\cap\hat A\lhd J(\hat A)$ which means that all inclusions are equalities: 
\begin{equation}
\label{eq.th4.5}
J(\hat A)=J(C)\cap J(\hat A)=J(C)\cap\hat A
\end{equation}
Analogously may be established that  
\begin{equation}
\label{eq.th4.6}
J(\check B)=J(C)\cap J(\check B)=J(C)\cap\check B.
\end{equation}

Let's denote by $\goth{S}(M)$ the socle of the left module  ${}_CM$ which is equal to 
\cite[Theorem  1]{Azumaya} the socle of the right module $M_Z$.
Because of the 
\cite[Theorem  7]{Azumaya} bimodule  ${}_{\bar C}\goth{S}(M)_{\bar Z}$ is a quasi-Frobenius. An isomorphism 
\begin{equation}
\label{eq.th4.7}
{}_{\bar C}\goth{S}(M)_{\bar Z} \cong {}_{\bar Z_{n,n}}{\bar Z}_{\bar
Z}^{(n)}.
\end{equation}
may be established by the same arguments as an isomorphism  (\ref{eq.th4.3+1}).
From  (\ref{eq.th4.7}) follows that  $\mathop{\rm\bf Z}(\bar C)=\bar Z$.

Let's show that rings  $\hat A$ and  $\check B$ are primary. Let 
$\hat A_1,\ldots,\hat A_t$ and  $\check B_1,\ldots,\check B_l$ are all primary components of rings 
 $\hat A$ and  $\check B$ respectively. 

Let  $\omega:C\to C/J(C)$ be a canonical epimorphism. Then  
\begin{equation}
\label{eq.th4.8}
\omega(C)=[\omega(\hat A),\omega(\check B)],
\end{equation}
whereas 
\begin{equation}
\label{eq.th4.9}
\omega(\hat A)\cong \hat A/(\hat A\cap J(C))=\hat A/J(\hat A)=\bar{\hat A}
\cong \bar A
\end{equation}
and 
\begin{equation}
\label{eq.th4.10}
\omega(\check B)\cong \check B/(\check B\cap J(C))=\check B/J(\check B)
=\bar{\check B} \cong {\bar B}^{op}.
\end{equation}
Hence because of the  \cite[Theorem III.9.2]{Jacobson}
\begin{equation}
\label{eq.th4.11}
\bar C=[\omega(\hat A),\omega(\check B)]=\sum_{i=\overline{1,t} \atop
j=\overline{1,l}} \oplus [\omega(\hat A_i),\omega(\check B_j)].
\end{equation}

Since $\bar C\cong {\bar Z}_{n,n}$  is a simple ring then between pair-wise orthogonal rings 
 $[\omega(\hat A_i),\omega(\check B_j)]$, $i=\overline{1,
t}, j=\overline{1,l}$, there is only one non-zero. Let's denote it
$[\omega(\hat A_1),\omega(\check B_1)]$, and for all  $i,j \ge2$
inclusions take place:  $\hat A_i, \check B_j < J(C)$. Hence every element of every of rings $\hat A_i, \check B_j$, $i,j \ge2$,  is a nil-potent what is impossible becaus every of those rings contains an idempotent. It is a contradiction which shows that   $t=l=1$ and rings $\hat A, \check B$ are primary.

Let's note now that because of the pare-wise commutativity of rings   $\omega(\hat A)$ and 
$\omega(\check B)$ inclusions take place: 
\begin{equation}
\label{eq.th4.12}
\bar Z=\mathop{\rm \bf Z}(\bar C)<\mathop{\rm \bf Z}(\omega(\hat A)),
\mathop{\rm\bf Z}(\omega(\check B))< \mathop{\rm\bf Z}(\bar C)=\bar Z.
\end{equation}
Hence  $\bar Z=\mathop{\rm\bf Z}(\omega(\hat A))=\mathop{\rm\bf Z}(
\omega(\check B))$. It means that rings  $\omega(\hat A)$ and 
$\omega(\check B)$ are central  $\bar Z$-algebras. 

Because rings  $\hat A, \check B$ are primary then isomorphisms take place: 
$\omega(\hat A)\cong \overline{\hat A}$, $\omega(\check B)\cong
\overline{\check B}$,
and rings  $\omega(\hat A)$, $\omega(\check B)$ are simple algebras. 

Since modules  ${}_{Z}\hat A$, ${}_Z\check B$ are finitely generated then 
 $\bar Z$-algebras  $\omega(\hat A)$, $\omega(\check B)$  are of finite dimension. 

So $\omega(\hat A), \omega(\check B)\in\goth{B}(\bar Z)$. According to Structure Theorem of Wedderburn--Artin there exist algebras with division (body) $\Delta,
\nabla$ and naturals  $r, s$ such that  $\omega(\hat A)=\Delta_{r,r}$,
$\omega(\check B)=\nabla_{s,s}$. Herewith  $\mathop{\rm\bf Z}(\Delta)=
\mathop{\rm\bf Z}(\nabla)=\bar Z$.

Because $\bar C=[\omega(\hat A),\omega(\check B)]$ and rings  $\omega(\hat
A)$, $\omega(\check B)$ are pair-wise commutative then the socle 
 $\goth{S}({}_CM)$ is  $(\omega(\hat A),\omega(\check B)^{op})$-bimodule. 
Consequently according to Statement  \ref{st2} there exists an epimorphism 
$\lambda:\omega(\hat A)\mathop{\otimes}\limits_{\bar Z}\omega(\check B)
\to\bar C$.
Herewith 
 $\ker \lambda=l_{
\omega(\hat A)\mathop{\otimes}\limits_{\bar Z}\omega(\check B)
}(\goth{S}(M))$ is a left annihilator of module  $\goth{S}(M)$ in ring 
$\omega(\hat A)\mathop{\otimes}\limits_{\bar Z}\omega(\check B)
$.
Because according to \cite[$\S12.4$, Proposition  b]{Pierce}
$\omega(\hat A)\mathop{\otimes}\limits_{\bar Z}\omega(\check B)
 \cong (\Delta \mathop{\otimes}\limits_{\bar Z}\nabla)_{rs,rs}$ is a simple algebra and left module 
${}_{\omega(\hat A)\mathop{\otimes}\limits_{\bar Z}\omega(\check B)
}\goth{S}({}_CM)$ is non-zero then   $\ker\lambda=\hat 0$. 
Thus 
\begin{equation}
\label{eq.th4.13}
\bar C\cong
(\Delta \mathop{\otimes}\limits_{\bar Z}\nabla)_{rs,rs}.
\end{equation}
Since  $\bar C\cong{\bar Z}_{n,n}$ then Brauer classes of equivalence  of 
$[\Delta\otimes\nabla]$ and  $[\bar Z]$ coincide i.e. classes  $[\Delta]$ and 
$[\nabla]$ are \cite[$\S12.5$, Proposition  a]{Pierce} mutually inverse elements of group 
 $\mathop{\rm\bf B}(\bar Z)$ and this equalities take place: 
$[\Delta]^{-1}=[\Delta^{op}]=[\nabla]$. Thereby  \cite[$\S12.5$, Proposition  b(ii)]{Pierce} 
 $\Delta^{op}\cong\nabla$.

Because of the  (\ref{eq.th4.9}) there exists an isomorphism $\mu:\omega(\hat A)\to\bar A$.
Since $\mathop{\rm\bf Z}(\omega(\hat A))=\bar Z$ then 
$\mathop{\rm\bf Z}(\bar A)=\mu(\bar Z)\cong \bar Z$. Consequently over an algebra  $\bar A$
may be established the structure of  $\bar Z$-module by have a put for all  $\bar a\in\bar A,
\bar z\in\bar Z$ $\bar z\cdot\bar a=\mu\left(\bar z\mu^{-1}(\bar a)\right)$.
Then  $\bar A$ is finite dimension central simple  $\bar Z$-algebra because it inherits properties of 
 $\bar Z$-algebra  $\omega(\hat A)$. Herewith basic algebras of $\bar A$ and  $\omega(\hat A)$ are isomorphic to algebra with division   $\Delta$.

Analogously because of (\ref{eq.th4.10}) there exists an isomorphism  $\nu:\omega(
\check B)^{op}\to\bar B$. As above $\mathop{\rm \bf Z}(\omega(\check B))
\cong \mathop{\rm\bf Z}(\omega(\check B)^{op})\cong\bar Z$ what mean 
$\mathop{\rm \bf Z}(\bar B)=\nu(\bar Z)\cong \bar Z$.
Structure of  $\bar Z$-module over  $\bar B$ is given by:
for all  $\bar b\in \bar B, \bar z\in \bar Z$ let's put 
$\bar z\cdot\bar b=\nu\left(\bar z\nu^{-1}(\bar b)\right)$.
Now  $\bar B$ inherits properties of  $\bar Z$-algebra $\omega(\check B)^{op}$ and 
becomes because of that finite dimensional central simple  $\bar Z$-algebra.
Because basic algebras of  $\bar B$ and  $\omega(\check B)^{op}$
are isomorphic to algebra with division  $\nabla^{op}$ then we have only to note that since 
 $\Delta\cong \nabla^{op}$ basic algebras of  $\bar Z$-algebras $\bar A$ and  $\bar B$ are isomorphic too.

                                                                     
I(d).\
Because of  \cite[$\S12.4$, Proposition  b(iv)]{Pierce} an isomorphism takes place:
\begin{equation}
\label{eq.th4.14}
\Delta \mathop{\otimes}\limits_{\bar Z}\nabla
\cong
\Delta \mathop{\otimes}\limits_{\bar Z}\Delta^{op}
\cong
{\bar Z}_{u,u},
\end{equation}
where  $u=\mathop{\rm dim}_{\bar Z}\Delta$.

Hence according to (\ref{eq.th4.13}) the equality takes place: 
\begin{equation}
\label{eq.th4.15}
\mathop{\rm dim}\nolimits_{\bar Z}\goth{S}(M)=r\cdot s\cdot\mathop{\rm
dim}\nolimits_{\bar
Z}\Delta.
\end{equation}

Now let's calculate dimension of the vector space $\goth{S}(M)_{\bar Z}$
in other way. 

Let  $e_{i,j}, i,j=\overline{1,r}$, $f_{u,v}, u,v=\overline{1,s}$ be a full systems of matrix units of rings  $\omega(\hat A)$ and  $\omega(\check B)$ respectively. Since 
 $\goth{S}(M)$ is $(\omega(\hat A),\omega(\check B)^{op})$-bimodule then two-sided Pierce decomposition of module  $\goth{S}(M)$ into direct sum of Abelian groups: 
\begin{equation}
\label{eq.th4.16}
\goth{S}(M)=\mathop{\sum_{i=\overline{1,r} \atop u=\overline{1,s}}\oplus}\
e_{i,i}\goth{S}(M)f_{u,u}.
\end{equation}
Every summand  $e_{i,i}\goth{S}(M)f_{u,u}$, $i=\overline{1,r},
u=\overline{1,s}$, is a   $(e_{i,i}\omega(\hat A)e_{i,i}, f_{u,u}
\omega(\check B)^{op}f_{u,u})$-bimodule. All these bimodules are pair-wise isomorphic and because  
$$
e_{i,i}\omega(\hat A)e_{i,i}\cong \Delta,\ \ \omega(\hat A)\cong\Delta_{r,r},
$$
$$
f_{u,u}\omega(\check B)^{op}f_{u,u}\cong\nabla^{op}\cong \Delta,\ \
\omega(\check B)^{op} \cong \Delta_{s,s},
$$
the isomorphism takes place: 
\begin{equation}
\label{eq.th4.17}
{}_{\omega(\hat A)}\goth{S}(M)_{\omega(\check B)^{op}}
\cong
{}_{\Delta_{r,r}}{\goth{S}_{r,s}}_{\Delta_{s,s}},
\end{equation}
where  $\goth{S}=e_{1,1}\goth{S}(M)f_{1,1}$.

From relation  (\ref{eq.th4.17}) equalities follow: 
\begin{equation}
\label{eq.th4.18}
\mathop{\rm dim}\nolimits_{\bar Z}\goth{S}(M)=
r\cdot s\cdot\mathop{\rm dim}\nolimits_{\bar Z}\Delta\cdot
\mathop{\rm dim}\nolimits_\Delta\goth{S}
=
r\cdot s\cdot\mathop{\rm dim}\nolimits_{\bar Z}\Delta\cdot
\mathop{\rm dim}\goth{S}_\Delta,
\end{equation}
where  $\mathop{\rm dim}_\Delta\goth{S}$ and  $\mathop{\rm dim}\goth{S}_\Delta$  are dimensions of left ${}_\Delta\goth{S}$ and right  $\goth{S}_\Delta$ vector spaces over the body  $\Delta$.

Comparing  (\ref{eq.th4.15}) and (\ref{eq.th4.18}) we find that 
\begin{equation}
\label{eq.th4.19}
\mathop{\rm dim}\nolimits_\Delta\goth{S}
=
\mathop{\rm dim}\goth{S}_\Delta
=1.
\end{equation}

Consequently 
\begin{equation}
\label{eq.th4.19+1}
{}_{\Delta}\goth{S}_{\Delta}\cong
{}_{\Delta}\Delta_{\Delta}
\end{equation}
is a quasi-Frobenius bimodule. Besides that from 
(\ref{eq.th4.19+1}) follows an existence of isomorphism 
\begin{equation}
\label{eq.th4.19+2}
{}_{\Delta_{r,r}}{\goth{S}_{r,s}}_{\Delta_{s,s}}
\cong
{}_{\Delta_{r,r}}{\Delta_{r,s}}_{\Delta_{s,s}},
\end{equation}
which implies an existence of isomorphism 
\begin{equation}
\label{eq.th4.19+3}
{}_{\omega(\hat A)}{\goth{S}({}_CM)}_{\omega(\check B)^{op}}
\cong
{}_{\Delta_{r,r}}{\Delta_{r,s}}_{\Delta_{s,s}}.
\end{equation}
Since  $\omega(\hat A)\cong \bar A$, $\omega(\check B)^{op}\cong ({\bar B}^{op})^{op}={\bar B}$ then from existence of isomorphism (\ref{eq.th4.19+3}) follows the existence of isomorphism 
\begin{equation}
\label{eq.th4.19+4}
{}_{\bar A}\goth{S}({}_CM)_{\bar B}
\cong
{}_{\Delta_{r,r}}{\Delta_{r,s}}_{\Delta_{s,s}}.
\end{equation}
Statement 
 \ref{st3}
shows that bimodule  
$
{}_{\Delta_{r,r}}{\Delta_{r,s}}_{\Delta_{s,s}}$  is a quasi-Frobenius. Consequently because of 
 (\ref{eq.th4.19+4}) bimodule  ${}_{\bar A}\goth{S}({}_C M)_{\bar B}$ is also quasi-Frobenius. 
                                                                            
So point  I of this Theorem is proven completely.

II.\ Now we can prove point  II of this Theorem. 
Let $\goth{S}({}_AM)=\goth{S}(M_B)=\goth{S}({}_CM)$.
As we already have established bimodule 
 ${}_{\bar A}\goth{S}({}_CM)_{\bar B}$
is quasi-Frobenius. Hence because of the Statement  \ref{st1} bimodule  ${}_AM_B$ is also quasi-Frobenius. 

Conversely if bimodule  ${}_AM_B$ is quasi-Frobenius then according to 
\cite[Теорема 1]{Azumaya} $\goth{S}({}_AM)=\goth{S}(M_B)$.
From other side 
$$
\goth{S}({}_AM)=\goth{S}({}_{\hat A}M)=r_M(J(\hat A)),
$$
$$
\goth{S}(M_B)=\goth{S}({}_{\check B}M)=r_M(J(\check B)). 
$$
Since  $J(\hat A), J(\check B)\lhd J(C)$ then 
$$
\goth{S}(M)=\goth{S}({}_AM)=\goth{S}(M_B)>\goth{S}({}_CM).
$$
Consequently quasi-Frobenius bimodule  ${}_{\bar A}\goth{S}({}_CM)_{\bar B}$ is subbimodule of quasi-Frobenius bimodule  ${}_{\bar A}\goth{S}(M) {}_{\bar B}$. This inclusion with necessity have to be an equality i.e. 
$$
\goth{S}({}_AM)=\goth{S}(M_B)=\goth{S}({}_CM).
$$

III.\ Since Brauer group $\mathop{\rm\bf B}(\bar Z)$ of the field  $\bar Z$ is trivial then 
 $\Delta\cong\nabla\cong\bar Z$.
Consequently 
 ${}_\Delta\goth{S}_\Delta\cong {}_{\bar Z}{\bar Z}_{\bar Z}$.
Because of above and  (\ref{eq.th4.19+4}) an isomorphism takes place: 
\begin{equation}
\label{eq.th4.19+5}
{}_{\bar A}\goth{S}({}_CM)_{\bar B}
\cong
{{}_{\bar Z_{r,r}}} {\bar Z_{r,s}}{}_{\bar Z_{s,s}}.
\end{equation}
Besides that from  (\ref{eq.th4.15}) the equality follows:  $n=r\cdot s$.

Let's show now that  $A\cong Z_{r,r}$, $B\cong Z_{s,s}$. Indeed rings 
$A$ and  $B$  contain as subrings $\phi(Z_{r,r})\cong
Z_{r,r}$ and  $\psi(Z_{s,s})\cong Z_{s,s}$ respectively. Let's denote by  $\hat\phi(Z_{s,s})$
and  $\check\psi(Z_{s,s})$ images of those subrings under converse mapping into  $C$.
Let's denote by  $\tilde C$ ring generated by  $\hat \phi(Z_{r,r})$
and  $\check\psi(Z_{s,s})$ in  $C$. According to Statement  \ref{st2} ring  $\tilde C$ is an epimorphic image of ring 
$$
\hat\phi(Z_{r,r}) \mathop{\otimes}\limits_{Z}\check\psi(Z_{s,s}) \cong Z_{r,r} 
\mathop{\otimes}\limits_Z Z_{s,s},
$$
i.e. 
\begin{equation}
\label{eq.th4.21}
\tilde C \cong
\hat\phi(Z_{r,r})\mathop{\otimes}\limits_{Z}\check\psi(Z_{s,s})/
l_{\hat\phi(Z_{r,r})\mathop{\otimes}\limits_{Z}\check\psi(Z_{s,s})}(M).
\end{equation}
Since $C\cong Z_{n,n}$ and  $n=rs$ then an isomorphism takes place: 
$$
C\cong Z_{r,r}\mathop{\otimes}_Z Z_{s,s},
$$
which implies isomorphisms: 
\begin{equation}
\label{eq.th4.21+1}
{}_CM
\cong
{}_{Z_{r,r}\mathop{\otimes}\limits_ZZ_{s,s}}M
\cong
{}_{\hat\phi(Z_{r,r})\mathop{\otimes}\limits_{Z}\check\psi(Z_{s,s})}M.
\end{equation}
Because left module  ${}_CM$ is faithful and according to (\ref{eq.th4.21+1}) left module 
$$
{}_{\hat\phi(Z_{r,r})\mathop{\otimes}\limits_{Z}\check\psi(Z_{s,s})}M
$$
is also faithful. In view of isomorphism  (\ref{eq.th4.21}) and isomorphisms 
 $\hat \phi(Z_{r,r})\cong Z_{r,r}$, $\check\psi(Z_{s,s})\cong Z_{s,s}$ we conclude that 
$\tilde C \cong C$.
Since 
$\tilde C=[\hat\phi(Z_{r,r}),\check\psi(Z_{s,s})]<[\hat A,\check B]=C$ then ring  $\tilde C$ is embedded identically into ring $C$ and previous isomorphism is only an equality. Consequently 
$\hat \phi(Z_{r,r})=\hat A$, $\check\psi(Z_{s,s})=\check B$, where from we get isomorphisms of point  III(a) of this Theorem.

Let's now trace a chain of isomorphisms: 
$$
{}_CM\cong{}_{A\mathop{\otimes}\limits_ZB^{op}}M\cong{}_AM_B.
$$
From other side
$$
{}_CM
\cong
{}_{Z_{rs,rs}}Q^{(rs)}
\cong
{}_{Z_{r,r}\mathop{\otimes}\limits_Z Z_{s,s}}Q^{(rs)}
\cong
{}_{Z_{r,r}}{Q_{r,s}}_{Z_{s,s}}.
$$
Here we have used a generalization of Proposition from  $\S10.1$ of 
\cite{Pierce}.
So we have established an isomorphism of point  III(b) of this Theorem.

Bimodule 
${}_{Z_{r,r}}{Q_{r,s}}_{Z_{s,s}}$ is quasi-Frobenius because of the Statement  \ref{st3}. Bimodule 
 ${}_AM_B$ is quasi-Frobenius because of the isomorphism of point 
III(b) of this Theorem.
\end{proof}

\section{Proof of the Theorem  \ref{TH4+1}}

For convenience let's remember that
the ring  $S$ is called \emph{ Galois-Eisenstein-Ore ring \/} (or \emph{ $GEO$-ring \/}) if it is finite completely primary ring of principal ideals i.e. the ring  $S$
contains unique (one-side) maximal ideal  $\goth{p}(S)$ and every one-side ideal of ring  $S$ is principal. 
\emph{Commutative\/} $GEO$-ring is called \emph{Galois-Eisenstein ring or $GE$-ring} \cite{NechaevMS}.
                                      
In arbitrary  $GEO$-ring  $S$ of characteristic $p^d$ contains 
\emph{ Galois subring \/} $R=GR(p^d,r)$ \cite{McDonald, Radghavendran}, $q=p^r$.
All such rings are conjugated in  $S$. Ring  $R$ is called \emph{coefficients ring \/} of  $S$.

And also let's repeat conditions of Theorem \ref{TH4+1}:

\begin{Theorem}
\label{TH4+3}

(a) For arbitrary  $GEO$-ring  $S$ bimodule ${}_SS_S$  is a quasi-Frobenius. 
                   
(b)
Let  $S$ be a  $GEO$-ring with a coefficients ring  $R=GR(p^d,r)$ and in addition  $d>1$ and
$(r,p)=1$.
Then  bimodule of translations 
 ${}_CS_Z$ of quasi-Frobenius 
$S$-bimodule 
$S$ is a quasi-Frobenius bimodule if and only if 
$S$ is a  $GE$-ring. 

\end{Theorem}

Let's denote by  $n(S)$ the nilpotency index of ideal  $\goth{p}(S)$.
The field  $S/\goth{p}(S)$ we denote by  $\bar S$ and an image of  $s\in S$  in 
$\bar S$  by $\bar s$.
Let  $\bar S=GF(q)$.

\emph{ For every finite completely primary ring   $S$
following statements are equivalent  \cite[Theorem  1.1, points  I, II, VI, VII,
VIII respectively ]{NechaevMS}:}

\emph{(a) $S$ is a $GEO$-ring;}

\emph{(b) $|S|=q^{n(S)}$;}

\emph{(c) Every one-side ideal of  $S$ is a degree of  $\goth{p}(S)$;}
     
\emph{(d) For every $\pi_t\in\goth{p}(S)^t\setminus\goth{p}(S)^{t+1}$,
$t=\overline{0,n(S)-1}$, this equalities take place: 
$\goth{p}(S)^t=\pi_tS=S\pi_t$ (if   $t=0$ then we denote $\goth{p}(S)^0=S$).}

\emph{(e) If  $\gamma: \bar S\to S$ is a mapping with property }
\begin{equation}
\label{eqms1}
\overline{\gamma(\bar s)}=\bar s\ \ \ \mbox{for all }\ \ \ s\in S,
\end{equation}
\emph{and  $\pi_t\in\goth{p}(S)^t\setminus\goth{p}(S)^{t+1}$,  
$t=\overline{0,n(S)-1}$, then an arbitrary element  $s\in S$ is uniquely represented in the form: 
 $s=\sum\limits_{t=0}^{n(S)-1}s_t\pi_t$, where   $s_t
\in \gamma(\bar S)$, $t=\overline{0,n(S)-1}$, and also in the form: 
$s=\sum\limits_{t=0}^{n(S)-1}\pi_ts_t'$, where  $s_t'
\in \gamma(\bar S)$, $t=\overline{0,n(S)-1}$. \/}

Let's view $S$ as left unitary $R$-module. System of elements 
 $s_1,\ldots,s_t\in {}_RS$ is called 
\cite[$\S2$]{NechaevMS} {\em free \/} if there no exists a non-trivial linear relation between these elements over $R$ and is called irreducible if no one of these elements is a linear combination over $R$ of others. 
Irreducible generating system of ${}_RS$ is called \emph{ basis \/} of ${}_RS$.
By \emph{dimension  \/} $\mathop{\rm dim}{}_RS$ is called a number of elements in basis of  ${}_RS$ which is equal to dimension of vector space $S/\goth{p}(R)S$ over field  $\bar R$.

Let $T({}_RS)=\{s\in S\ |\ \goth{p}(R)^{n(R)-1}s=0\}$ be a set of elements of ${}_RS$ annihilated by some non-zero element from  $R$. Then \cite[Proposition 2.1]{NechaevMS} the number of elements in maximal free subsystem of  ${}_RS$ satisfies to equality: 
 $\mathop{\rm rang}{}_RS=\mathop{\rm dim}_{\bar R}(S/T({}_RS))$.

It is known  \cite[Corollary  2.3]{NechaevMS} that the equality takes place: 
$\mathop{\rm dim}{}_RS=
\mathop{\rm dim}S_R$,
$\mathop{\rm rang}{}_RS=
\mathop{\rm rang}S_R$,
and besides that an arbitrary basis of  ${}_RS$ is also a basis of  $S_R$ and conversely.
Because of that we can speak simply about rank, dimension and basis of  $S$ over  $R$.

For some natural $e=e(S|R)$  \cite[Theorem  1.1]{NechaevMS} these equalities take place: 
$$
\goth{p}(R)\cdot S=S\cdot\goth{p}(R)\cdot S=S\cdot\goth{p}(R)
$$ and 
$$
\goth{p}(R)\cdot S=\goth{p}(S)^e.
$$ 
Besides that $n(S)=(n(R)-1)e+\rho$ where  $\rho=\rho(S|R)$ satisfies to inequalities:  $1\le\rho\le e$.
Herewith   \cite[Proposition  5.1]{NechaevMS} \emph{ $S$ is a  $R$-bimodule of the rank 
$\rho=\rho(S|R)$ and dimension  $e=e(S|R)$  whereas $T(S)=\goth{p}(S)^\rho$.\
\/}

Let  $\sigma\in\mathop{\rm Aut}(R)$. \emph{ By the ring of Ore polynomials \/}
$R[x,\sigma]$ is called ring of polynomials   with usual addition and multiplication established by the rule:
$x\cdot r=r^\sigma\cdot x, r\in R$. If  $t=t(\sigma)$ is an order of  $\sigma$,
then center  $\mathop{\rm \bf Z}(R[x,\sigma])$ of ring  $R[x,\sigma]$ contains of and only of polynomials of the form  $c_0+c_t x^t+\cdots+c_{tm}x^{tm}$ where  $c_0,\ldots,c_{tm}
\in R_\sigma=\{ r\in R\ |\ r^\sigma=r \}$.

Polynomial  $x^l 
+c_{l -1}x^{l -1}+\cdots+c_0\in R[x,\sigma]$ is  called \emph{
Eisenstein polynomial \/} if  $c_i\in \goth{p}(R)$ for 
$i=\overline{0,l -1}$ and  $c_0\not\in\goth{p}(R)^2$ if  $n(R)>1$.
Eisenstein polynomial of the form 
\begin{equation}
\label{eq.ms.5.1}
c(x)=x^{tm}+c_{t(m-1)}x^{t(m-1)}+\cdots+c_0
\end{equation}
is called \emph{ special \/} if either 
\begin{equation}
\label{eq.ms.5.2}
c(x)\in\mathop{\rm \bf Z}(R[x,\sigma]),
\end{equation}
either for some  $a<m$
\begin{equation}
\label{eq.ms.5.3}
c(x)-c_{ta}x^{ta}\in\mathop{\rm \bf Z}(R[x,\sigma]),\ \
c_{ta}^\sigma-c_{ta}\in\goth{p}(R)^{n(R)-1}\setminus\{0\}.
\end{equation}

It is known  \cite[Theorem  5.2 II]{NechaevMS} that \emph{ if  $R$ is a ring of coefficients of $GEO$-ring  $S$, $n(R)>1$, $e=e(S|R)$, $\rho=\rho(S|R)$,  then there exist an automorphism $\sigma\in\mathop{\rm Aut}(R)$ such that $t=t(\sigma)$ divides 
 $e$, $e=t\cdot m$, and special Eisenstein polynomial of the form 
(\ref{eq.ms.5.1}) with the property: 
\begin{equation}
\label{eq.ms.5.5}
S\cong R[x,\sigma]/I, 
\end{equation}
where  $I=c(x)R[x,\sigma]+x^\rho\goth{p}(R)^{n(R)-1}[x,\sigma]$. Herewith
(\ref{eq.ms.5.3}) is fulfilled  only if $\rho=ta+1$.\/}

\emph{ Automorphism  $\sigma$  in
(\ref{eq.ms.5.5}) is uniquely determined by the ring  $S$ and do  not depends on the choice of coefficient ring $R$\/}
\cite[Proposition  5.5]{NechaevMS}.

Let's remember that  $GEO$-ring  $S$ is a \emph{ Galois--Eisenstein ring 
($GE$-ring )\/} iff 
$t=t(\sigma)=1$ i.e. Galois--Eisenstein ring is nothing but commutative Galois--Eisenstein--Ore ring  \cite{NechaevMZ}.

Besides that  \cite[Proposition 5.7]{NechaevMS} \emph{if  $\pi$ is a 
$\sigma$-element from  $\goth{p}(S)\setminus\goth{p}(S)^2$ ( i.e. for every  $r\in R$ this equality satisfies: $\pi r=r^\sigma\pi$) then: }

\emph{ I) The centralizer $C(R)$ of ring  $R$ in  $S$ is a   $GE$-ring  (\cite{NechaevMZ})
\begin{equation}
\label{eq.ms.5.9}
C(R)=R[\pi^t]=R+R\pi^t+\cdots+R\pi^{t(m-1)},
\end{equation}
whereas $e(C(R)|R)=m$, $\rho(C(R)|R)=\rho'$,  $\rho'=\left[\frac{\rho}{t}\right]$, and 
\begin{equation}
\label{eq.ms.5.10}
C(R)\cong R[y]/z(y)R[y]+y^{\rho'}\goth{p}(R)^{n(R)-1}[y],
\end{equation}
whereas  $c(x)=z(x^t)$.}

\emph{ II) The center  $\mathop{\rm \bf Z}(S)$ is equal to 
\begin{equation}
\label{eq.ms.5.11}
C(R)_\sigma=R_\sigma+R_\sigma\pi^t+\cdots+R_\sigma\pi^{t(m-1)},
\end{equation}
whereas  $\rho\not\equiv1(\mathop{\rm mod} t)$ and  is equal to 
$$
C(R)_\sigma+\goth{p}(S)^{n(S)-1}
$$  in other case.  \/}
                   
It is known \cite[$\S13$]{Kasch} that arbitrary left-side and right-side together Artinian ring  $T$ 
with identity is  \emph{ quasi-Frobenius \/} if and only if $T$-bimodule 
$T$ is a quasi-Frobenius.

$\mathop{\rm\bf Proof\  of\  the\  Theorem\  \ref{TH4+1}}$.
   
(a)
Let's use the Statement  \ref{st1}. These equalities take place: 
\begin{equation}
\label{eq.ms.1+1}
\begin{array}{c}
\goth{S}({}_SS)=r_S(\goth{p}(S))=\goth{p}(S)^{n(S)-1}=
(\pi S)^{(n(R)-1)e+\rho-1}=
\\
=\pi^{\rho-1}\cdot (\pi S)^{(n(R)-1)e}
=\pi^{\rho-1}\cdot\goth{p}(R)^{n(R)-1},
\end{array}
\end{equation}
whereas   $\pi\in\goth{p}(S)\setminus\goth{p}(S)^2$.
Analogously
\begin{equation}
\label{eq.ms.1+2}
\goth{S}(S_S)=l_S(\goth{p}(S))
=\goth{p}(S)^{(n(R)-1)e+\rho-1}=
\pi^{\rho-1}
\cdot
\goth{p}(R)^{n(R)-1}.
\end{equation}
This way  $\goth{S}({}_SS)=\goth{S}(S_S)=\goth{S}(S)$.
Besides that $\mathop{\rm dim}\goth{S}(S)_{\bar S}=\mathop{\rm dim}_{\bar S}\goth{S}(S)=1$.
Hence according to point (4) of Theorem  \ref{th1} bimodule 
${}_{\bar S}\goth{S}(S)_{\bar S}$ is quasi-Frobenius. 
Thus bimodule  ${}_SS_S$ and ring  $S$ is also quasi-Frobenius.

(b)
Let's note that  $n(R)=d$. Because  $d>1$ the ring  $S$ has a form: 
\begin{equation}
\label{eq.ms.5.5+1}
S\cong R[x,\sigma]/I,
\end{equation}
where 
$\rho=\rho(S|R)$,
$e=e(S|R)$, $t=t(\sigma)$,
$\sigma\in\mathop{\rm Aut}(R)$,
$I=c(x)R[x,\sigma]+x^\rho\goth{p}(R)^{n(R)-1}[x,\sigma]$,
$c(x)=x^{tm}+c_{t(m-1)}x^{t(m-1)}+\cdots+c_0
$
is a special Eisenstein polynomial, whereas condition 
$$
c(x)-c_{ta}x^{ta}\in\mathop{\rm \bf Z}(R[x,\sigma]),\ \
c_{ta}^\sigma-c_{ta}\in\goth{p}(R)^{n(R)-1}\setminus\{0\} 
$$
is satisfied only if 
 $\rho=ta+1$. 

Besides that by $r=\log_pq$ we have  $t|r$. Since  under condition  $(r,p)=1$ we have then  $(t,p)=1$.

Let 
$\cal R$ be a ring of translations of  $R$-bimodule $S$,
$C$ be a ring of translations of  $S$-bimodule  $S$.

Let's prove that ring  $\cal R$ is isomorphic to foreign direct sum of  
$\min\{t,\rho\}$ 
copies of ring   $R=GR(p^d,r)$ and 
$\max\{0,t-\rho\}$ 
copies of ring  $\tilde R
=R/\goth{p}(R)^{n(R)-1}=GR(p^{d-1},r)$. 

Let's denote by  $\epsilon_l\in{\cal R}$, $l=\overline{0,t-1}$, 
endomorphisms of  $R$-bimodule 
 $S$ of the form: 
\begin{equation}
\label{eq.th4+1.0-1+2}
\epsilon_l: S\to 
S^{(l)}=\sum_{j=0}^{e-1}\delta_{l,j\pmod{t}}R\pi^j,
\end{equation}
operating according to the rule:
\begin{equation}
\label{eq.2005.1}
\epsilon_l\left(\sum_{j=0}^{e-1} r_j\pi^j\right)=
\sum_{j=\overline{0,e-1} \atop j\equiv l\pmod{t}}
r_j\pi^j,
\end{equation}
where  $\delta_{\lambda,\mu}$ is a Kronecker symbol.

Since equalities take place:  $S=\sum_{j=0}^{e-1}
R\pi^j$ and  $c(\pi)=0$
where 
$\pi$ is a  $\sigma$-element from  $\goth{p}(S)
\setminus\goth{p}(S)^2$ we have: 
\begin{equation}
\label{eq.th4+1.0+5}
\hat S=\sum_{j=0}^{e-1}\hat R\hat\pi^j,
\ \
\check S=\sum_{j=0}^{e-1}\check R\check\pi^j.
\end{equation}
Hence 
\begin{equation}
\label{eq.th4+1.0+7}
C=[\hat S,\check S]=\left[\sum_{i=0}^{e-1}\hat R\hat\pi^i,
\sum_{j=0}^{e-1}\check R\check\pi^j\right]=\sum_{i,j=\overline{0,e-1}}
\left(\sum_{r_i,r_j\in R}\hat r_i\check r_j\right)\hat\pi^i\check\pi^j.
\end{equation}
It is evident that ring of translations  $\cal R$ of  $R$-bimodule 
$S$ may be represented in the form: 
${\cal R}=\sum\limits_{r_1,r_2\in R}\hat r_1\check r_2$.
Let's operate the structure of this set. 
To do this let's describe previously the ring  $\check R$.

Let  $r\in R$ and $s=\sum_{l=\overline{0,e-1}}r_l\pi^l\in S$. Then 
\begin{equation}
\label{eq.th4+1.0+8}
\begin{array}{c}
\check r(s)=s\cdot r=\sum_{l=\overline{0,e-1}}r^{\sigma^l}r_l\pi_l=
\sum_{l=\overline{0,e-1}}\widehat{(r^{\sigma^l})}(r_l\pi^l)=\\
=\sum_{l=\overline{0,t-1}}\widehat{(r^{\sigma^l})}\left(r_l\pi^l+r_{l+t}\pi^{l
+t}+\cdots+r_{l+t(m-1)}\pi^{l+t(m-1)}\right),
\end{array}
\end{equation}

since  $r^{\sigma^l}=r^{\sigma^j}$ if and only if 
$l\equiv j(\mathop{\rm mod} t)$.

Define mapping  $\phi_l: R\to\mathop{\rm End}(S)$, $l=\overline{
0,t-1}$, by the rule: for every  $ r \in  R$ put 
\begin{equation}
\label{eq.th4+1.0+9}
\phi_l(
r)\left(\sum_{j=\overline{0,e-1}}r_j
\pi^j\right)=
\sum_{j=\overline{0,t-1}}
\delta_{l,j(\mathop{\rm mod}t)}\cdot
\widehat{(r^{\sigma^l})}\left(r_j 
\pi^j\right),
\end{equation}
i.e. 
\begin{equation}
\label{eq.th4+1.0+9+1}
\phi_l(r)(S)=\widehat{(r^{\sigma^l})}(S^{(l)}).
\end{equation}

Thus 
\begin{equation}
\label{eq.th4+1.0+10}
\check r(s)=\sum_{l=\overline{0,t-1}}\phi_l(r)(s).
\end{equation}

Let's show that the family of set  $\phi_l(R)$, $l=\overline{0,t-1}$ are pair-wise orthogonal subrings of the ring  $\mathop{\rm End}(S)$.
To do this, it suffices to prove that in fact there are elements
$\epsilon_l\in{\cal R}$, $l=\overline{0,t-1}$, with pointed in equality 
(\ref{eq.th4+1.0-1+2}) property which it is convenient to represent in the form 
\begin{equation}
\label{eq.th4+1.0+10+1}
\epsilon_l(S^{(j)})=\delta_{l,j}S^{(l)}, \ \ l,j=\overline{0,t-1}.
\end{equation}

It is evident that an arbitrary mapping  $\beta\in {\cal R}$ may be represented in the following form: 
\begin{equation}
\label{eq.th4+1.0+10+2}
\beta(S)=\sum_{l=0}^{t-1}\hat\beta_l(S^{(l)}).
\end{equation}
We pose the mapping $\beta$ into compliance  a vector 
$\vec \beta=(\beta_0,\ldots,\beta_{t-1})$, $\beta_l\in R, l=\overline{0,t-1}
$, and determine the action of this  vector on the set
$S$ in this way: 
\begin{equation}
\label{eq.th4+1.0+10+3}
\vec\beta(S)=(\beta_0,\ldots,\beta_{t-1})\cdot
\left(
\begin{array}{c}
S^{(0)} \\
\cdots \\
S^{(t-1)} \\
\end{array}
\right)=
\sum_{l=0}^{t-1} \beta_l(S^{(l)})=\beta(S).
\end{equation}
Let  $\beta,\gamma\in{\cal R}$, $r\in R$. Then mapping  $\zeta=\hat
r\cdot \beta\in {\cal R}$ is represented by vector  $\vec\zeta=r\cdot \vec\beta$, and mapping 
 $\xi=\beta+\hat r\cdot \gamma\in {\cal R}$ is represented by vector   $\vec\xi=\vec\beta+r\cdot\vec\gamma$. Converse is also true. 

Let's remember that for  $r\in R$
\begin{equation}
\label{eq.th4+1.0+10+4}
\vec{\check r}=(r,r^\sigma,\ldots,r^{\sigma^{t-1}}).
\end{equation}
Let's denote this vector as  $\vec\beta(r)$ and consider the matrix 
\begin{equation}
\label{eq.th4+1.0+10+5}
C(r)=\left(
\begin{array}{c}
\vec\beta(r)\\
\vec\beta(r^\sigma)\\
\cdots\\
\vec\beta(r^{\sigma^{t-1}})\\
\end{array}
\right)=
\left(
\begin{array}{cccc}
r,&r^\sigma,&\ldots,&r^{\sigma^{t-1}}\\
r^\sigma,&r^{\sigma^2},&\ldots,&r \\
\cdots&\cdots&\cdots&\cdots\\
r^{\sigma^{t-1}},& r,&\ldots,& r^{\sigma^{t-2}}\\
\end{array}
\right).
\end{equation}
                            
Since $(t,p)=1$ over field  $\bar R=GF(q)$ there are exactly $t$ different roots from unit of degree 
$t$ namely:   $\aleph_1,\ldots,\aleph_t$. Let  $f_r(x)=
\bar r+\overline{r^\sigma}\cdot x+\cdots+\overline{r^{\sigma^{t-1}}}\cdot
x^{t-1}\in\bar R[x]$. Then there is an equality: 
$$
\mathop{\rm det}\overline{C(r)}=(-1)^{\frac{(t-1)(t-2)}{2}}\cdot\prod_{j=1}^t
f_r(\aleph_j).
$$
When $r\in R^*\setminus R_\sigma$ the right-hand side of previous equality differs from zero. Since with 
 $r\in R^*\setminus R_\sigma$ and  $(t,p)=1$ the matrix  $C(r)$ is invertible. 
Hence by equivalent transformations of rows matrix  $C(r)$ is possible to lead to the identity matrix
 $E_{t\times t}$.
It remains to remark that for every  $l\in\overline{0,t-1}$ row  $\vec
E_{l+1}$ is a vector of desired mapping 
$\epsilon_l\in{\cal R}$.

It is easy to see that there are equalities: 
\begin{equation}
\label{eq.th4+1.0+10+6}
\phi_l(r)=\check r\cdot\epsilon_l=\widehat{r^{\sigma^l}}\cdot\epsilon_l,\ \
l=\overline{0,t-1},
\end{equation}
from which it follows that $\phi_l(R)$, $l=\overline{0,t-1}$, are pair-wise orthogonal subrings of ring  $\cal R$. 

Besides that now  the equality is evident:
\begin{equation}
\label{eq.th4+1.0+10+7}
{\cal R}=\mathop{\sum\oplus}_{l=0}^{t-1} \phi_l(R).
\end{equation}

Let's describe now rings  $\phi_l(R)$, $l=\overline{0,t-1}$. To do this let's note that for arbitrary 
$r\in R$ $\phi_l(r)=\check 0$
if and only if  $r^{\sigma^l}\cdot\pi^l=0$, i.e. 
$r^{\sigma^l}\pi^l\in\goth{p}(S)^{n(S)}$, whereas   $n(S)=e(n(R)-1)+\rho$,
$0\le\rho\le e-1$. Let  $r\in \goth{p}(R)^\kappa$.
Then $r^{\sigma^l}\pi^l\in\goth{p}(S)^{e\kappa+l}$ and equality 
$r^{\sigma^l}\pi^l=0$ is achieved if and only if 
$e\kappa+l\ge e(n(R)-1)+\rho$, i.e.  $l\in\overline{\rho,e-1}$,
$\kappa=n(R)-1$. Let's note now  that  $t\le e$.
Hence these isomorphisms take place: 
\begin{equation}
\label{eq.th4+1.0+11}
\phi_l(R)\cong R,\ l=\overline{0,\min\{t,\rho\}-1},
\end{equation}
\begin{equation}
\label{eq.th4+1.0+12}
\phi_l(R)\cong R/\goth{p}(R)^{n(R)-1}=GR(q^{d-1},p^{d-1}),\
l=\overline{\rho,t-1}.
\end{equation}

Let's describe now the ring  $\bar C=C/J(C)$.
Let's note that  $C={\cal R}[\hat\pi,\check\pi]$.
Since  $\hat\pi$ and  $\check\pi$ are nil-potent elements of ring $C$ then 
 $\hat\pi,\check\pi\in J(C)$.
Consequently 
 $\bar C=C/J(C)\cong\mathop{\sum\oplus}\limits_{l=\overline{0,t-1}}
\overline{\phi_l(R)}$ is isomorphic to foreign direct sum of  $t$ copies of the field 
 $\bar R=GF(q)$.

Immediately from the definitions it follows that for arbitrary  $GEO$-ring $S$
common center  $Z$ of rings  $\hat S$ and  $\check S$ is equal to  $\mathop{\rm\bf Z}(S)$.

Let's suppose now that bimodule 
 ${}_CS_Z$ is quasi-Frobenius. Then since  $Z$ is local Artinian ring we can apply results of Theorem 
\ref{TH4}. In particular according to point  I(a) of Theorem \ref{TH4} the ring  $C$ is isomorphic to 
 $Z_{n,n}$ whereas 
$n=\mathop{\rm dim}\nolimits_{\bar Z} \goth{S}(S)$.
Let's remark that 
$\bar Z=\overline{R_\sigma}$ and 
$\mathop{\rm dim}_{\bar Z}\goth{S}(S)=t=\mathop{\rm ord}\sigma$.

Thus ring  $\bar C=C/J(C)\cong\mathop{\sum\oplus}\limits_{l=\overline{0,t-1}}
\overline{\phi_l(R)}$ is isomorphic to foreign direct sum of $t$ copies of the field 
 $\bar R=GF(q)$ and has to be isomorphic to ring of  $t\times t$-matrices over ring  ${R_\sigma}$.

But the ring  $\bar C$ is isomorphic to matrix ring only if 
 $t=1$ i.e. if 
$S$ is a $GE$-ring.

Converse implication is evident because firstly an arbitrary
$GE$-ring is quasi-Frobenius and secondly for arbitrary commutative ring  $S$
ring of translations of  $S$-bimodule  $S$ and common center of  $\hat S$ and  $\check S$ relatively to 
 $S$ coincide with  $S$.
{\hfill $\Box$}

\end{document}